\documentclass[12pt,a4paper,english]{article}
\usepackage[affil-it]{authblk}
\usepackage[utf8]{inputenc}
\usepackage[margin=1in]{geometry}
\usepackage[english]{babel}
\fontfamily{phv}\selectfont
\usepackage[sort,numbers]{natbib}

\pdfoutput=1

\usepackage{todonotes, ulem}
\usepackage[utf8]{inputenc}
\usepackage[english]{babel}

\usepackage{array}
\usepackage{amsfonts}
\usepackage{amssymb}
\usepackage{amsthm}
\usepackage{mathtools}
\usepackage{mathtext}
\usepackage{mathrsfs}
\usepackage{verbatim}
\usepackage{color}
\usepackage{makeidx}
\usepackage{float}
\usepackage[hyphens]{url}
\usepackage{enumerate}
\usepackage[hidelinks]{hyperref}

\usepackage[affil-it]{authblk}

\DeclareMathOperator*{\essinf}{ess\,inf}
\DeclareMathOperator*{\esssup}{ess\,sup}
\DeclareMathOperator*{\argmin}{arg\,min}
\DeclareMathOperator*{\argmax}{arg\,max}

\newtheorem{theorem}{Theorem}[section]
\newtheorem{lemma}[theorem]{Lemma}

\newtheorem{proposition}[theorem]{Proposition}

\theoremstyle{definition}
\newtheorem{definition}{Definition}
\newtheorem{assumption}{Assumption}

\theoremstyle{remark}
\newtheorem{remark}[theorem]{Remark}

\allowdisplaybreaks[1]

\makeatletter  
\@addtoreset{equation}{section}

\makeatother  
\newcommand{\msc}[1]{\textbf{MSC2010 Classification:} #1.}
\newcommand{\jel}[1]{\textbf{JEL Classification:} #1.}
\newcommand{\keywords}[1]{\textbf{Key words:} #1.}

\begin{document}
	\title{\textbf{A probabilistic verification theorem for the finite horizon two-player zero-sum optimal switching game in continuous time}}
	\author[1]{Said Hamad\`{e}ne}
	\author[2]{Randall Martyr\thanks{Corresponding author. Email: r.martyr@qmul.ac.uk}\thanks{Financial support received from the EPSRC via grant EP/N013492/1.}}
	\author[2]{John Moriarty\thanks{Financial support received from the EPSRC via grant EP/P002625/1.}}
	\affil[1]{Universit\'{e} du Maine, LMM, Avenue Olivier Messiaen, 72085 Le Mans, Cedex 9, France.}
	\affil[2]{School of Mathematical Sciences, Queen Mary University of London, Mile End Road, London E1 4NS, United Kingdom.}
	\maketitle
	\begin{abstract}
		In this paper we study continuous-time two-player zero-sum optimal switching games on a finite horizon. Using the theory of doubly reflected BSDEs with interconnected barriers, we show that this game has a value and an equilibrium in the players' switching controls.
	\end{abstract}

\msc{91A15, 91A55, 91A05, 93E20, 60G40, 49N25}
\vspace{+4pt}

\jel{C61, C72, C73}
\vspace{+8pt}

\noindent\keywords{optimal switching, optimal switching games, stopping times, optimal stopping problems, optimal stopping games, backward stochastic differential equations} 

\section{Zero-sum optimal switching game}
Optimal switching is a generalisation of optimal stopping which has various applications in economics and mathematical finance. It consists of one or more agents and a system which they control by successively switching the system's operational mode according to a discrete set of choices. There are several works on optimal switching problems in continuous time, and a survey of the literature identifies two main approaches: an analytical approach using partial differential equations (PDEs) and a probabilistic one. 

Methods based on PDEs and associated variational inequalities appeared as early as the 1970s, under the topic of impulsive control for diffusion processes (see \cite{Robin1976} and the references therein). A viscosity solutions approach to this type of PDE appeared in the late 1980s to early 1990s (for instance, \cite{Tang1993}) and is still the topic of active research \cite{Lundstrom2013a}.

Probabilistic solution methods were being applied since the 1970s and 1980s in various degrees of generality (see \cite{Bismut1979,Morimoto1987,Robin1976,Zabczyk1973,Zvonkin1971} for instance), and most of the recent research in this area has been a combination of the martingale approach via Snell envelopes (\cite{Djehiche2009,Martyr2014b}) and the theory of backward stochastic differential equations (BSDE) (\cite{Chassagneux2011,Elie2014,hu-tang2}).

All of the aforementioned references are concerned with single-person optimisation problems. Multiple-person optimal switching problems in a stochastic setting, the topic under which the present work falls, have been studied less frequently in the literature (there is related work for deterministic systems such as \cite{Yong1990,Yong1990a}). In the zero-sum setting there are previous works in continuous time including \cite{Tang2007,hu-tang,Djehiche2017} and, if viewed as a special case of impulse control games, \cite{Stettner1982b,Cosso2013}.

From the probabilistic point of view, the zero-sum switching game leads to the study of the following system of reflected BSDEs with inter-connected bilateral obstacles: Find a system of processes $(Y^{i,j},Z^{i,j},K^{i,j})_{(i,j)\in \Gamma}$ such that for any $(i,j)\in \Gamma$ and $s \in [0,T]$,  
\begin{equation}\label{eq:DRBSDE-Bilteral-Interconnected-Obstacles}
\begin{split}
(i) \quad & Y^{i,j}_{s} = h^{i,j} + \int_{s}^{T}f^{i,j}_{t}{d}t + K^{i,j}_{T} - K^{i,j}_{s} - \int_{s}^{T}Z^{i,j}_{t}{d}B_{t}\,; \\
(ii) \quad & Y^{i,j}_{s} \le U^{i,j}_{s}(\boldsymbol{Y})\;\text{and}\;Y^{i,j}_{s} \ge L^{i,j}_{s}(\boldsymbol{Y}) \,; \\
(iii) \quad & \int_{s}^{T}(Y^{i,j}_{t} - U^{i,j}_{t}(\boldsymbol{Y})){d}K^{i,j,-}_{t} = \int_{s}^{T}(L^{i,j}_{t}(\boldsymbol{Y}) - Y^{i,j}_{t}){d}K^{i,j,+}_{t} = 0,
\end{split}
\end{equation}
where: i) $\Gamma^1=\{1,...m_1\}$ and $\Gamma^2=\{1,...m_2\}$ are operational modes controlled by players 1 and 2 respectively, and $\Gamma=\Gamma^1\times \Gamma^2$ is the set of operational modes for the system; 
ii) $f^{i,j}$, $\hat{g}^{i,k}$ and $\check{g}^{j,\ell}$ (resp.  $h^{i,j}$) are given stochastic processes (resp. random variables) defining the game's economic data; iii) $K^{i,j}$ is a finite variation process and $K^{i,j,\pm}$ are the increasing processes in its orthogonal decomposition $K^{i,j} \coloneqq K^{i,j,+} - K^{i,j,-}$; iv) $\boldsymbol{Y}=(Y^{i,j})_{(i,j)\in \Gamma}$, $L^{i,j}_s(\boldsymbol{Y}) \coloneqq \max_{k\neq i, k\in \Gamma^1}\{Y^{k,j}_s-\hat{g}^{i,k}_{s}\}$ and 
$U^{i,j}_s(\boldsymbol{Y})\coloneqq \min_{\ell \neq j, j\in \Gamma^2}\{Y^{i,\ell}_s+\check{g}^{j,\ell}_{s}\}$.  

In the diffusion framework, randomness stems from an $\mathbb{R}^{k}$-valued diffusion process $X^{s,x} \coloneqq (X_t^{s,x})_{s \le t \le T}$, $(s,x) \in [0,T] \times \mathbb{R}^{k}$ is fixed, which satisfies: 
\begin{equation}\label{sdex}
\begin{cases}
\forall t\in [s,T],\,\,\, X_t^{s,x}=x+\int_{s}^{t}b(r,X^{s,x}_r)dr+\int_{s}^{t}\sigma(r,X^{s,x}_r)dB_r;\\
X_r^{s,x}=x,\,\,r\in [0,s].
\end{cases}
\end{equation}
In this setting, the system \eqref{eq:DRBSDE-Bilteral-Interconnected-Obstacles} is connected to the following {\it Hamilton-Jacobi-Bellman} system of PDEs with obstacles: For any $(i,j)\in \Gamma$ and $(s,x) \in [0,T] \times \mathbb{R}^{k}$,
\begin{equation} \label{mainsyst-vi}
\begin{cases}
   \min\Bigl\{v^{i,j}(s,x)-L^{i,j}(\boldsymbol{v})(s,x)\,,\, \max \bigl\{ v^{i,j}(s,x) -U^{i,j}(\boldsymbol{v})(s,x),\\
   \qquad \qquad \qquad \qquad \qquad \qquad \qquad \quad \bigl(-\partial_s-\mbox{L}^{X}\bigr)(v^{i,j})(s,x)-f^{i,j}(s,x) \bigr\} \Bigr\} = 0\, ;\\
v^{i,j}(T,x)=h^{i,j}(x),
\end{cases}
\end{equation} where $\mbox{L}^X$ is the generator associated with $X^{s,x}$; $\boldsymbol{v} = (v^{i,j})_{(i,j) \in \Gamma}$, $L^{i,j}(\boldsymbol{v}) \coloneqq \max_{k\neq i, k\in \Gamma^1}\{v^{k,j}-\hat{g}^{i,k}\}$ and $U^{i,j}(\boldsymbol{v}) \coloneqq \min_{\ell \neq j, j\in \Gamma^2}\{v^{i,\ell}+\check{g}^{j,\ell}\}$; $f^{i,j}$, $h^{i,j}$, $\hat{g}^{i,k}$ and $\check{g}^{j,\ell}$ are deterministic functions.

System \eqref{eq:DRBSDE-Bilteral-Interconnected-Obstacles} is studied, for example, in \cite{hu-tang} where it is shown that the solution exists when $\check{g}^{j,\ell}$ and $\hat{g}^{i,k}$ are constant. In the diffusion framework, it is also studied in \cite{Djehiche2017} where the authors have shown that the solution exists and is unique under rather mild regularity assumptions on the data. The connection, through the usual Feynman-Kac formula, with the viscosity solution to the system of PDEs \eqref{mainsyst-vi} is also established. However, the interpretation of $Y^{i,j}$ (or $v^{i,j}$) as the value of the underlying zero-sum switching game, as might be expected, is stated only in the case when $f^{i,j}$ and $h^{i,j}$ are {\it separated} with respect to $i$ and $j$: $f^{i,j}=f_1^i+f_2^j$ and $h^{i,j}=h_1^i+h_2^j$. The case when $f^{i,j}$ or $h^{i,j}$ are not separated is still open, and its resolution is the main objective of this work. In particular, we show that the solution of system \eqref{eq:DRBSDE-Bilteral-Interconnected-Obstacles}, when it exists (for example, in the diffusion framework), coincides with the value function of the zero-sum switching game. As a result, the unique viscosity solution to the system of PDEs \eqref{mainsyst-vi} coincides with the value function of the associated zero-sum switching game.  

This paper is organised as follows. In Section \ref{Section:Notation}, we introduce the zero-sum switching game. In Section \ref{Section:Verification}, we show the main result, that $Y^{i,j}$ coincides with the value of the zero-sum game. Moreover, we provide results on the existence of optimal strategies in the game. For completeness, we also interpret our findings in the diffusion framework.

\section{Probabilistic setup and notation}\label{Section:Notation}
We follow closely the setup in \cite{Djehiche2017}, working on a finite horizon $[0,T]$ and filtered probability space $(\Omega,\mathcal{F},\mathbb{F},\mathbb{P})$ where $\mathbb{F} = (\mathcal{F}_{t})_{0 \le t \le T}$ is the usual completion of the natural filtration of $B = (B_{t})_{0 \le t \le T}$, a $d$-dimensional standard Brownian motion.
\begin{itemize}
	\item Let $\mathcal{T}$ be the set of $\mathbb{F}$-stopping times bounded above by $T$, and for a given $\nu \in \mathcal{T}$, $\mathcal{T}_{\nu}$ the set of all $\tau \in \mathcal{T}$ satisfying $\tau \ge \nu$ a.s.
	\item For any sub-$\sigma$-algebra $\hat{\mathcal{F}}$ of $\mathcal{F}$, let $L^{p}(\hat{\mathcal{F}})$, $1 \le p < \infty$, denote the set of $p$-integrable $\hat{\mathcal{F}}$-measurable random variables, and set $L^{p} \coloneqq L^{p}(\mathcal{F})$.
	\item Let $\mathcal{H}^{2}$ be the set of $\mathbb{F}$-progressively measurable processes $w = (w_{t})_{0 \le t \le T}$ satisfying, $$\mathbb{E}\left[\int_{0}^{T}(w_{t})^{2}{d}t\right] < +\infty.$$
	\item Let $\mathcal{S}^{2}$ be the set of $\mathbb{F}$-adapted processes $w = (w_{t})_{0 \le t \le T}$ with paths that are right-continuous with left limits satisfying,
	\[
	\sup_{0 \le t \le T}|w_{t}| \in L^{2}.
	\]
	Let $\mathcal{S}^{2}_{c} \subset \mathcal{S}^{2}$ denote the subset of processes $w \in \mathcal{S}^{2}$ with continuous paths.
	\item Let $\mathcal{K}^{2}$ denote the set of $\mathbb{F}$-adapted right-continuous with left limits processes $K$ of finite variation satisfying $K_{0} = 0$ and,
	\[
	 \int_{0}^{T}|{d}K_{t}| \in L^{2},
	\]
	where $|{d}K_{t}(\omega)|$ is the total variation measure on $[0,T]$. Let $\mathcal{K}^{2}_{c}$ denote the subset of processes $K \in \mathcal{K}^{2}$ with continuous paths.
\end{itemize}

\begin{definition}\label{def:si}
	Let $Y$ be a right-continuous with left limits semi-martingale having decomposition $Y_{t} = Y_{0} + M_{t} + K_{t}$ where $M$ is a local martingale, $K$ has finite variation, and $M_{0} = K_{0} = 0$. Note that $M$ is continuous due to the choice of filtration $\mathbb{F}$ (see Lemma 14.5.2 of \cite{Cohen2015}). We say that $Y$ is square-integrable and write $Y \in \mathcal{W}^{2}$ if
	$$Y_{0} \in L^{2}, M \in \mathcal{S}^{2} \text{ and } K \in \mathcal{K}^{2}.$$
	If $Y$ is continuous then we write $Y \in \mathcal{W}^{2}_{c}$.
\end{definition}

Let $\mathcal{S}^{2,m}_{c}$ denote the $m$-product of $\mathcal{S}^{2}_{c}$. Similarly we define $\mathcal{H}^{2,m}$, $L^{2,m}$, $\mathcal{S}^{2,m}$, $\mathcal{K}^{2,m}$, \ldots, for the $m$-products of the spaces $\mathcal{H}^{2}$, $L^{2}$, $\mathcal{S}^{2}$, $\mathcal{K}^{2}$, and so on.

\subsection{Costs, rewards and switching controls}

Let $\Gamma^{k} = \{1,\ldots,m_{k}\}$, $k \in \{1,2\}$, be a finite, discrete set representing the operating modes that player $k$ can choose. Let $\Gamma = \Gamma^{1} \times \Gamma^{2}$ denote the product space of operating modes $\gamma=(\gamma^{(1)},\gamma^{(2)})$, having cardinality $|\Gamma| = m = m_{1} \times m_{2}$.

\begin{itemize}
	\item For $(i,j) \in \Gamma$, $f^{i,j} \in \mathcal{H}^{2}$ defines a running reward paid by player 2 to player 1 and $h^{i,j} \in L^{2}(\mathcal{F}_{T})$ a terminal reward paid by player 2 to player 1, when player 1's (resp. player 2's) active mode is $i$ (resp. $j$).
	\item For $i_{1},i_{2} \in \Gamma^{1}$, $\hat{g}^{i_{1},i_{2}} \in \mathcal{S}^{2}_{c}$ defines a non-negative payment from player 1 to player 2 when the former switches from $i_{1}$ to $i_{2}$.
	\item For $j_{1},j_{2} \in \Gamma^{2}$, $\check{g}^{j_{1},j_{2}} \in \mathcal{S}^{2}_{c}$ defines a non-negative payment from player 2 to player 1 when the former switches from mode $j_{1}$ to $j_{2}$.
\end{itemize}

For all $(i,j) \in \Gamma$ and $t \in [0,T]$ we set $\hat{g}^{i,i}_{t} = \check{g}^{j,j}_{t} = 0$.

\subsubsection{Individual switching controls and strategies}

\begin{definition}[Switching controls and their indicator functions]\label{Definition:Switching-Controls}
	A \it{control} for player $1$ is a sequence $\alpha = (\sigma_{n},\xi_{n})_{n \ge 0}$ such that,
	\begin{enumerate}
		\item for all $n \ge 0$, $\sigma_{n} \in \mathcal{T}$ and is such that $\sigma_{n} \le \sigma_{n+1}$, $\mathbb{P}$-a.s., and $\mathbb{P}(\{\sigma_{n} < T\; \forall n \ge 0\}) = 0$;
		\item for all $n \ge 0$, $\xi_{n}$ is an $\mathcal{F}_{\sigma_{n}}$-measurable $\Gamma^{1}$-valued random variable;
		\item for $n \ge 1$, on $\{\sigma_{n} < T\}$ we have $\sigma_{n} < \sigma_{n+1}$ and $\xi_{n} \neq \xi_{n-1}$, while on $\{\sigma_{n} = T\}$ we have $\xi_{n} = \xi_{n-1}$.
	\end{enumerate}
	Let $\mathsf{A}$ denote the set of controls for player $1$. The set $\mathsf{B}$ of controls $\beta = \left(\tau_{n},\zeta_{n}\right)_{n \ge 0}$ for player $2$, where the $\zeta_{n}$ are $\Gamma^{2}$-valued, is defined analogously. Denoting by $C^{\alpha}_{N}$ the cost of the first $N \ge 1$ switches,
	\[
	C^{\alpha}_{N} \coloneqq \sum_{n=1}^{N}\hat{g}^{\xi_{n-1},\xi_{n}}_{\sigma_{n}},
	\]
	note that the limit $\lim\limits_{N \to \infty}C^{\alpha}_{N}$ is well defined.
\end{definition}

\begin{definition}
	A control $\alpha \in \mathsf{A}$ for player 1 is said to be \it{square-integrable} if,
	\[
	\lim_{N \to \infty}C^{\alpha}_{N} \in L^{2}.
	\]
	Let $\mathcal{A}$ denote the set of such controls. Similarly, the set $\mathcal{B}$ of square-integrable controls for player 2 consists of those $\beta \in \mathsf{B}$ satisfying,
	\[
	\lim_{N \to \infty}C^{\beta}_{N} \in L^{2},
	\]
	where
	\[
	C^{\beta}_{N} \coloneqq \sum_{n=1}^{N}\check{g}^{\zeta_{n-1},\zeta_{n}}_{\tau_{n}}.
	\]
\end{definition}

\begin{definition}[Non-anticipative switching strategies]\label{Definition:Switching-Strategies}
	Let $s \in [0,T]$ and $\nu \in \mathcal{T}_{s}$. Two controls $\alpha^{1},\alpha^{2} \in \mathsf{A}$ with $\alpha^{1} = \bigl(\sigma^{1}_{n},\xi^{1}_{n}\bigr)_{n \ge 0}$ and $\alpha^{2} = \bigl(\sigma^{2}_{n},\xi^{2}_{n}\bigr)_{n \ge 0}$ are said to be {\it equivalent}, denoting this by $\alpha^{1} \equiv \alpha^{2}$, {\it on $[s,\nu]$} if we have a.s.,
	\[
	\xi^{1}_{0}\mathbf{1}_{[\sigma^{1}_{0},\sigma^{1}_{1}]}(t) + \sum\limits_{n \ge 1}\xi^{1}_{n}\mathbf{1}_{(\sigma^{1}_{n},\sigma^{1}_{n+1}]}(t) = \xi^{2}_{0}\mathbf{1}_{[\sigma^{2}_{0},\sigma^{2}_{1}]}(t) + \sum\limits_{n \ge 1}\xi^{2}_{n}\mathbf{1}_{(\sigma^{2}_{n},\sigma^{2}_{n+1}]}(t), \quad s \le t \le \nu.
	\]
	A {\it non-anticipative strategy} for player $1$ is a mapping $\overline{\alpha} \colon \mathsf{B} \to \mathsf{A}$ such that:
	\begin{itemize}
		\item {\it Non-anticipativity:} for any $s \in [0,T]$, $\nu \in \mathcal{T}_{s}$, and $\beta^{1},\beta^{2} \in \mathsf{B}$ such that $\beta^{1} \equiv \beta^{2}$ on $[s,\nu]$, we have $\overline{\alpha}(\beta^{1}) \equiv \overline{\alpha}(\beta^{2})$ on $[s,\nu]$.
		\item {\it Square-integrability:} for any $\beta \in \mathcal{B}$ we have $\overline{\alpha}(\beta) \in \mathcal{A}$.
	\end{itemize}
	In a similar manner we define non-anticipative strategies for player 2. Let $\mathscr{A}$ and $\mathscr{B}$ denote the set of non-anticipative strategies for players 1 and 2 respectively.
\end{definition}
\begin{definition}
	For $s \in [0,T]$ and $i \in \Gamma^{1}$, let $\mathsf{A}_{s}^{i}$ denote the set of controls $\alpha \in \mathsf{A}$ satisfying $\xi_{0} = i$ and $\sigma_{0} = s$. Similarly, define $\mathsf{B}_{s}^{j}$ for $s \in [0,T]$ and $j \in \Gamma^{2}$. Analogous notation will be used below for other classes of controls, for example square-integrable controls $\mathcal{A}_{s}^{i}$, $\mathcal{B}_{s}^{j}$, and strategies $\mathscr{A}^{i}_{s}$, $\mathscr{B}^{j}_{s}$.
\end{definition}

\subsubsection{Coupling of controls}

We now define the coupling of two controls $\alpha \in \mathsf{A}$ and $\beta \in \mathsf{B}$ under the following assumption: player 1's switch is implemented first if both players decide to switch at the same instant.

\begin{definition}\label{Definition:Coupled-Control}
	Given controls $\alpha \in \mathsf{A}$ and $\beta \in \mathsf{B}$, define the {\it coupling} $\gamma(\alpha,\beta) = (\rho_{n},\gamma_{n})_{n \ge 0}$ where $\rho_{n} \in \mathcal{T}$ is defined by,
	\begin{equation}\label{eq:coupled-switching-times}
	\rho_{n} = \sigma_{r_{n}} \wedge \tau_{s_{n}},
	\end{equation}
	with $r_{0} = s_{0} = 0$, $r_{1} = s_{1} = 1$ and for $n \ge 2$,
	\[
	r_{n} = r_{n-1} + \mathbf{1}_{\{\sigma_{r_{n-1}} \le\, \tau_{s_{n-1}}\}}, \quad s_{n} = s_{n-1} +  \mathbf{1}_{\{\tau_{s_{n-1}} <\, \sigma_{r_{n-1}}\}},
	\]
	and $\gamma_{n}$ is a $\Gamma$-valued random variable such that
	$\gamma_{0} = (\xi_{0},\zeta_{0})$ and for $n \ge 1$,
	\begin{equation}\label{eq:coupled-switching-modes}
	\gamma_{n} = \begin{cases}
	\bigl(\xi_{r_{n}},\gamma_{n-1}^{(2)}\bigr), & \text{on}\; \{\sigma_{r_{n}} \le \tau_{s_{n}},\; \sigma_{r_{n}} < T\} \\
	\bigl(\gamma_{n-1}^{(1)},\zeta_{s_{n}}\bigr), & \text{on}\; \{\tau_{s_{n}} < \sigma_{r_{n}} \} \\
	\gamma_{n-1}, & \text{on}\; \{\tau_{s_{n}} = \sigma_{r_{n}} = T\}.
	\end{cases}
	\end{equation}
	Define for all $0 \le t \le T$,
	\begin{equation}\label{eq:Optimal-Switching-Indicator-Continuous-Time}
	u_{t} = \gamma_{0}\mathbf{1}_{[\rho_{0},\rho_{1}]}(t) + \sum\limits_{n \ge 1}\gamma_{n}\mathbf{1}_{(\rho_{n},\rho_{n+1}]}(t),
	\end{equation}
\end{definition}
\noindent where $(\rho_{n},\rho_{n+1}]=\emptyset$ on $\{\rho_{n} = \rho_{n+1}\}$.

Note that the coupling $\gamma(\alpha,\beta) = (\rho_{n},\gamma_{n})_{n \ge 0}$ of the controls $\alpha \in \mathsf{A}^{i}_{s}$ and $\beta \in \mathsf{B}^{j}_{s}$ has the following properties:
\begin{enumerate}
	\item $\rho_{0} = s$ and for all $n \ge 0$ we have $\rho_{n} \in \mathcal{T}$ and $\rho_{n} \le \rho_{n+1}$  $\mathbb{P}$-a.s., and $\mathbb{P}(\{\rho_{n} < T\; \forall n \ge 0\}) = 0$;
	\item $\gamma_{0} = (i,j)$ and for all $n \ge 0$ the random variable $\gamma_{n}$ is $\mathcal{F}_{\rho_{n}}$-measurable, $\Gamma$-valued and $\gamma_{n+1} \neq \gamma_{n}$ on $\{\rho_{n+1} < T\}$.
\end{enumerate}
Write $C^{\gamma(\alpha,\beta)}_{N}$ for the joint cumulative cost of the first $N$ switches,
	\[
	C^{\gamma(\alpha,\beta)}_{N} =  \sum_{n=1}^{N}\Bigl[\,\hat{g}^{\gamma_{n-1}^{(1)},\gamma_{n}^{(1)}}_{\rho_{n}} - \check{g}^{\gamma_{n-1}^{(2)},\gamma_{n}^{(2)}}_{\rho_{n}}\Bigr], \qquad N \ge 1.
	\]	
\begin{definition}\label{Definition:Joint-Admissibility}
The coupling $\gamma(\alpha,\beta) = (\rho_{n},\gamma_{n})_{n \ge 0}$ of the controls $\alpha \in \mathsf{A}^{i}_{s}$ and $\beta \in \mathsf{B}^{j}_{s}$ is said to be {\it admissible}, writing $\gamma(\alpha,\beta) \in \mathcal{G}^{i,j}_{s}$ to indicate this, if
$\sup_{N \ge 1}\bigl|C^{\gamma(\alpha,\beta)}_{N}\bigr| \in L^{2}$. 
\end{definition}

Note that for every $\alpha \in \mathsf{A}$ and $\beta \in \mathsf{B}$ we have $\lim\limits_{N \to \infty}C^{\gamma(\alpha,\beta)}_{N} = \lim\limits_{N \to \infty}C^{\alpha}_{N} - \lim\limits_{N \to \infty}C^{\beta}_{N}$. Using the triangle inequality, we see that every pair of square-integrable controls $(\alpha,\beta)$, $\alpha \in \mathcal{A}^{i}_{s}$ and $\beta \in \mathcal{B}^{j}_{s}$, satisfies $\gamma(\alpha,\beta) \in \mathcal{G}^{i,j}_{s}$.

\subsection{The zero-sum switching game}\label{Section:Zero-Sum-Example}

For the zero-sum game we assume that player $1$ is the maximiser and define the total reward from its perspective. Letting $(s,i,j) \in [0,T) \times \Gamma$ be the initial state and recalling \eqref{eq:Optimal-Switching-Indicator-Continuous-Time}, we have
\begin{multline}
J_{s}^{i,j}(\gamma(\alpha,\beta)) = \mathbb{E}\left[\int_{s}^{T}f^{u_{t}}_{t}{d}t - \sum_{n=1}^{\infty}\Bigl[\,\hat{g}^{\gamma_{n-1}^{(1)},\gamma_{n}^{(1)}}_{\rho_{n}} - \check{g}^{\gamma_{n-1}^{(2)},\gamma_{n}^{(2)}}_{\rho_{n}}\Bigr] + h^{u_{T}} \Bigm \vert \mathcal{F}_{s}\right], \\
\quad \alpha \in \mathsf{A}^{i}_{s}, \enskip \beta \in \mathsf{B}^{j}_{s}.\label{eq:Zero-Sum-Performance-Functional}
\end{multline}

The {\it lower and upper values} for this game, denoted respectively by $\check{V}^{i,j}_{s}$ and $\hat{V}^{i,j}_{s}$, are defined as follows:
\begin{equation}\label{Definition:Optimal-Switching-Game-Lower-Upper-Values}
\begin{cases}
\check{V}^{i,j}_{s} \coloneqq \esssup\limits_{\alpha \in \mathcal{A}_{s}^{i}}\essinf\limits_{\beta \in \mathcal{B}_{s}^{j}}J^{i,j}_{s}(\gamma(\alpha,\beta)) \\
\hat{V}^{i,j}_{s} \coloneqq \essinf\limits_{\beta \in \mathcal{B}_{s}^{j}}\esssup\limits_{\alpha \in \mathcal{A}_{s}^{i}}J^{i,j}_{s}(\gamma(\alpha,\beta)).
\end{cases}
\end{equation}
Note that $\check{V}^{i,j}_{s} \le \hat{V}^{i,j}_{s}$ a.s.

\begin{definition}\label{Definition:Switching-Game-Value}
	The game is said to have a {\it value at $(s,i,j)$} if 
	\begin{equation}\label{eq:Switching-Game-Value}
	\check{V}^{i,j}_{s} = \hat{V}^{i,j}_{s} \quad \text{a.s.}
	\end{equation}
	The common value $V^{i,j}_{s}$, when it exists, is referred to as the game's {\it solution at $(s,i,j)$}. When $s = T$ we formally set $\check{V}^{i,j}_{T} = \hat{V}^{i,j}_{T} = h^{i,j}$.
\end{definition}

In this paper we construct a pair of controls $(\alpha^{*},\beta^{*}) \in \mathsf{A}^{i}_{s} \times \mathsf{B}^{j}_{s}$ such that $\gamma(\alpha^{*},\beta^{*}) \in \mathcal{G}^{i,j}_{s}$ and the game has a value $V^{i,j}_{s} = J^{i,j}_{s}(\gamma(\alpha^{*},\beta^{*}))$ (see Theorem~\ref{Theorem:Verification} below). Such a result was obtained in \cite{Djehiche2017} under the assumption $f^{i,j} = f^{i}_{1} + f^{j}_{2}$, $h^{i,j} = h^{i}_{1} + h^{j}_{2}$ for $(i,j) \in \Gamma$. Our result is obtained by dynamic programming and the connection between doubly reflected backward stochastic differential equations (DRBSDEs) with implicitly defined barriers and zero-sum optimal stopping games. We also prove the existence of optimal non-anticipative strategies $\overline{\alpha^{*}} \in \mathscr{A}^{i}_{s}$ and $\overline{\beta^{*}} \in \mathscr{B}^{j}_{s}$ which are {\it robust} in the sense that each is a best response to the worst-case opponent.

\subsection{Assumptions}

\begin{definition}\label{Definition:Loop}
	For $N \ge 2$ a {\it loop} in $\Gamma$ of length $N-1$ is a sequence $\{(i_{1},j_{1}),\ldots,(i_{N},j_{N})\}$ of elements in $\Gamma$ with $N-1$ distinct members such that $(i_{N},j_{N}) = (i_{1},j_{1})$ and either $i_{q+1} = i_{q}$ or $j_{q+1} = j_{q}$ for any $q = 1,\ldots,N-1$.
\end{definition}
 
Throughout this paper we make the following assumptions, which are closely related to those in \cite{Djehiche2017}:
\begin{assumption}\label{Assumption:Switching-Costs}
	We impose the following conditions on the switching costs:
	\begin{enumerate}
		\item {\it Non-negativity}: $\min\limits_{i_{1} \in \Gamma^{1}}\hat{g}^{i,i_{1}} \ge 0$ and $\min\limits_{j_{1} \in \Gamma^{2}}\check{g}^{j,j_{1}} \ge 0$ for all $i \in \Gamma^{1}$, $j \in \Gamma^{2}$.
		\item{\it Consistency:}  
		\begin{enumerate}
\item For all sequences $\{i_{1},i_{2},i_{3}\} \in \Gamma^{1}$ and $\{j_{1},j_{2},j_{3}\} \in \Gamma^{2}$ with $i_{1} \neq i_{2}$, $i_{2} \neq i_{3}$ and $j_{1} \neq j_{2}$, $j_{2} \neq j_{3}$, we have for all $t \in [0,T]$,
		\begin{equation}\label{eq:SwitchingCostNoArbitrage-1}
		\hat{g}^{i_1,i_3}_{t} < \hat{g}^{i_1,i_2}_{t} + \hat{g}^{i_2,i_3}_{t}\;\; \mathbb{P}\text{-a.s.} \enskip \text{and} \enskip \check{g}^{j_1,j_3}_{t} < \check{g}^{j_1,j_2}_{t} + \check{g}^{j_2,j_3}_{t}\;\; \mathbb{P}\text{-a.s.}
		\end{equation}

\item For all $(i,j) \in \Gamma$ we have,
		\begin{equation}
		\max_{i_{1} \neq i, i_{1} \in \Gamma^{1}}\{h^{i_{1},j} - \hat{g}^{i,i_{1}}_{T}\} \le h^{i,j} \le \min_{j_{1} \neq j, j_{1} \in \Gamma^{2}}\{h^{i,j_{1}} + \check{g}^{j,j_{1}}_{T}\} \;\; \mathbb{P}\text{-a.s.}
		\end{equation}
	\end{enumerate}
			\item {\it Non-free loop property:}	For any loop $\{(i_{1},j_{1}),\ldots,(i_{N},j_{N})\}$ in $\Gamma$ we have for all $t \in [0,T]$,
		\begin{equation}\label{eq:SwitchingCostNoArbitrage-2}
		\sum_{q = 1}^{N-1}\varphi^{q,q+1}_{t} \neq 0 \;\;\mathbb{P}\text{-a.s.},
		\end{equation}
		where
		$\varphi^{q,q+1}_{t} = -\hat{g}^{i_{q},i_{q+1}}_{t}\mathbf{1}_{\{i_{q} \neq i_{q+1}\}} + \check{g}^{j_{q},j_{q+1}}_{t}\mathbf{1}_{\{j_{q} \neq j_{q+1}\}}$.
		\end{enumerate}
		\end{assumption}

\section{A probabilistic verification theorem for the zero-sum game}\label{Section:Verification}

Theorem~\ref{Theorem:Verification} uses the system \eqref{eq:DRBSDE-Bilteral-Interconnected-Obstacles} to prove the existence of a value for the zero-sum game. Recall that $m = |\Gamma|$ is the number of joint operating modes $(i,j) \in \Gamma$. For $(i,j) \in \Gamma$ define the {\it lower and upper switching operators}, $L^{i,j} \colon \mathcal{S}^{2,m}_{c} \to \mathcal{S}^{2}_{c}$ and $U^{i,j} \colon \mathcal{S}^{2,m}_{c} \to \mathcal{S}^{2}_{c}$ respectively, as follows: for $\boldsymbol{Y} \in \mathcal{S}^{2,m}_{c}$,
\begin{equation}\label{eq:Implicit-Obstacles}
\begin{cases}
L^{i,j}(\boldsymbol{Y}) = \max\limits_{i_{1} \neq i, i_{1} \in \Gamma^{1}}\{Y^{i_{1},j} - \hat{g}^{i,i_{1}}\}, \\
U^{i,j}(\boldsymbol{Y}) = \min\limits_{j_{1} \neq j, j_{1} \in \Gamma^{2}}\{Y^{i,j_{1}} + \check{g}^{j,j_{1}}\}.
\end{cases}
\end{equation}
Let $\boldsymbol{L} \colon \mathcal{S}^{2,m}_{c} \to \mathcal{S}^{2,m}_{c}$ and $\boldsymbol{U} \colon \mathcal{S}^{2,m}_{c} \to \mathcal{S}^{2,m}_{c}$ be the operators defined, using matrix notation, by $\boldsymbol{L} = (L^{i,j})_{(i,j) \in \Gamma}$ and $\boldsymbol{U} = (U^{i,j})_{(i,j) \in \Gamma}$. The following definition formalises the concept of a solution to \eqref{eq:DRBSDE-Bilteral-Interconnected-Obstacles}.

\begin{definition}\label{Definition:DRBSDE}
A {\it solution} to the system of DRBSDEs with terminal value $\boldsymbol{h} \in L^{2,m}(\mathcal{F}_{T})$, driver $\boldsymbol{f} \in \mathcal{H}^{2,m}$, and implicit barriers $\boldsymbol{L}$ and $\boldsymbol{U}$, is a triple $(\boldsymbol{Y},\boldsymbol{Z},\boldsymbol{K}) \in \mathcal{S}^{2,m}_{c} \times \mathcal{H}^{2,m} \times \mathcal{K}^{2,m}_{c}$ such that a.s. for all $(i,j) \in \Gamma$ and all $0 \le s \le T$,
\begin{equation}\tag{\ref{eq:DRBSDE-Bilteral-Interconnected-Obstacles} revisited}
\begin{split}
(i) \quad & Y^{i,j}_{s} = h^{i,j} + \int_{s}^{T}f^{i,j}_{t}{d}t + K^{i,j}_{T} - K^{i,j}_{s} - \int_{s}^{T}Z^{i,j}_{t}{d}B_{t}\,; \\
(ii) \quad & Y^{i,j}_{s} \le U^{i,j}_{s}(\boldsymbol{Y})\;\text{and}\;Y^{i,j}_{s} \ge L^{i,j}_{s}(\boldsymbol{Y})\,; \\
(iii) \quad & \int_{s}^{T}(Y^{i,j}_{t} - U^{i,j}_{t}(\boldsymbol{Y})){d}K^{i,j,-}_{t} = \int_{s}^{T}(L^{i,j}_{t}(\boldsymbol{Y}) - Y^{i,j}_{t}){d}K^{i,j,+}_{t} = 0,
\end{split}
\end{equation}
where $K^{i,j,+}$ and $K^{i,j,-}$ are the increasing processes in the orthogonal decomposition $K^{i,j} \coloneqq K^{i,j,+} - K^{i,j,-}$.
\end{definition}

Note that for any solution to \eqref{eq:DRBSDE-Bilteral-Interconnected-Obstacles}, the stochastic integral $\int_{0}^{t}Z^{i,j}_{s}{d}B_{s}$ is well-defined, and is a martingale belonging to $\mathcal{S}^{2}_{c}$ (see Chapter 3 of \cite{Chung2014}).
\begin{theorem}\label{Theorem:Verification}
	Suppose there exists a solution $(\boldsymbol{Y},\boldsymbol{Z},\boldsymbol{K})$ to the DRBSDE \eqref{eq:DRBSDE-Bilteral-Interconnected-Obstacles}. For every initial state $(s,i,j) \in [0,T] \times \Gamma$,
	\begin{enumerate}[(i)]
		\item {\it Existence of value:} the switching game has a value with,
		\begin{equation}\label{eq:Verification-Ansatz}
		Y^{i,j}_{s} = V^{i,j}_{s}\;\;\text{a.s.}
		\end{equation}
		\item {\it Existence of optimal controls:} there exists a pair of controls $(\alpha^{*},\beta^{*}) \in \mathsf{A}^{i}_{s} \times \mathsf{B}^{j}_{s}$ such that $\gamma(\alpha^{*},\beta^{*}) \in \mathcal{G}^{i,j}_{s}$ and $V^{i,j}_{s} = J^{i,j}_{s}(\gamma(\alpha^{*},\beta^{*}))$ a.s.
		\item {\it Existence of optimal strategies:} there exist non-anticipative strategies $\overline{\alpha^{*}} \in \mathscr{A}^{i}_{s}$ and $\overline{\beta^{*}}  \in \mathscr{B}^{j}_{s}$ that are optimal in the robust sense:
		\[
		\begin{cases}
		\essinf\limits_{\beta \in \mathcal{B}^{j}_{s}} J^{i,j}_{s}\bigl(\gamma(\overline{\alpha^{*}}(\beta),\beta)\bigr) = \esssup\limits_{\overline{\alpha} \in \mathscr{A}^{i}_{s}}\essinf\limits_{\beta \in \mathcal{B}^{j}_{s}} J^{i,j}_{s}\bigl(\gamma(\overline{\alpha}(\beta),\beta)\bigr) \\
		\esssup\limits_{\alpha \in \mathcal{A}^{i}_{s}} J^{i,j}_{s}\bigl(\gamma(\alpha,\overline{\beta^{*}}(\alpha))\bigr) = \essinf\limits_{\overline{\beta} \in \mathscr{B}^{j}_{s}}\esssup\limits_{\alpha \in \mathcal{A}^{i}_{s}} J^{i,j}_{s}\bigl(\gamma(\alpha,\overline{\beta}(\alpha))\bigr)
		\end{cases}
		\]
		Furthermore, these robust values are equal to the game's value,
		\[
		\esssup\limits_{\overline{\alpha} \in \mathscr{A}^{i}_{s}}\essinf\limits_{\beta \in \mathcal{B}^{j}_{s}} J^{i,j}_{s}\bigl(\gamma(\overline{\alpha}(\beta),\beta)\bigr) = V^{i,j}_{s} = \essinf\limits_{\overline{\beta} \in \mathscr{B}^{j}_{s}}\esssup\limits_{\alpha \in \mathcal{A}^{i}_{s}} J^{i,j}_{s}\bigl(\gamma(\alpha,\overline{\beta}(\alpha))\bigr).
		\]
	\end{enumerate}
\end{theorem}

This concept of robustness, which is well known in the optimal control and differential games literature \cite{Krasovskii1988, Ball2002, Bayraktar2016robust}, is natural in the context of zero-sum games \cite{Cosso2013}.

\begin{remark}
	Since the switching costs are non-negative we get the following type of {\it Mokobodski's condition}: there exists a system of processes $\boldsymbol{w} = \{w^{i,j}\}_{(i,j) \in \Gamma}$ belonging to $\mathcal{W}^{2,m}_{c}$ such that for all $(i,j) \in \Gamma$: for all $0 \le t \le T$ a.s.,
	\begin{equation}\label{eq:Mokobodski-Condition}
	\max_{i_{1} \neq i, i_{1} \in \Gamma^{1}}\{w^{i_{1},j}_{t} - \hat{g}^{i,i_{1}}_{t}\} \le w^{i,j}_{t} \le \min_{j_{1} \neq j, j_{1} \in \Gamma^{2}}\{w^{i,j_{1}}_{t} + \check{g}^{j,j_{1}}_{t}\}.
	\end{equation}
	Indeed, by taking $\boldsymbol{w}$ to be the $m$-dimensional null process, $\boldsymbol{w} \equiv \boldsymbol{0}$, it is easily verified that $\boldsymbol{w} \in \mathcal{W}^{2,m}_{c}$ and \eqref{eq:Mokobodski-Condition} holds. Mokobodski's condition \eqref{eq:Mokobodski-Condition} is an extension of that typically assumed for single-agent switching problems in a variety of settings \cite{Bismut1981,Bouchard2009,Elie2014,Martyr2014b}, or for two-player Dynkin games or DRBSDEs \cite{Bismut1977,Hamadene2006,Martyr2014c,Pham2013,Dumitrescu2014}, both of which are special, somewhat degenerate, cases of the optimal switching game studied here.
	
	Let us point out that for any solution $(\boldsymbol{Y},\boldsymbol{Z},\boldsymbol{K})$ to the DRBSDE \eqref{eq:DRBSDE-Bilteral-Interconnected-Obstacles}, $\boldsymbol{Y}$ satisfies Mokobodski's condition \eqref{eq:Mokobodski-Condition} and, a posteriori, also belongs to $\mathcal{W}^{2,m}_{c}$. Condition \eqref{eq:Mokobodski-Condition} can therefore be seen as a feasibility check for the inequality constraint \eqref{eq:DRBSDE-Bilteral-Interconnected-Obstacles}--$(ii)$: there exists at least one system of processes $\boldsymbol{Y}$ which satisfies \eqref{eq:DRBSDE-Bilteral-Interconnected-Obstacles}--$(ii)$ within a suitable class of candidates. Actually, we know from the results in \cite{Pham2013} that well-posedness of \eqref{eq:DRBSDE-Bilteral-Interconnected-Obstacles} is intricately linked to Mokobodski's condition \eqref{eq:Mokobodski-Condition}.
\end{remark}

\subsection{Proof of Theorem~\ref{Theorem:Verification}}
The existence of a solution to the DRBSDE \eqref{eq:DRBSDE-Bilteral-Interconnected-Obstacles} is closely related to the existence of both a value and a Nash equilibrium in the following Dynkin game (see for example \cite{Hamadene2006, Dumitrescu2014, Hamadene2004}, and also \cite{Stettner1982b} for the relation to impulse control games with delay).

\begin{proposition}\label{Problem:Implicit-Dynkin-Game-Upper-Value}
	Suppose there exists a solution $(\boldsymbol{Y},\boldsymbol{Z},\boldsymbol{K})$ to the DRBSDE \eqref{eq:DRBSDE-Bilteral-Interconnected-Obstacles}. Then for all $(s,i,j) \in [0,T] \times \Gamma$ a.s.:
	\vskip1em
	$(a)$ \enskip
	\begin{equation}\label{eq:Implicit-Dynkin-Game-Value}
	Y^{i,j}_{s} = \essinf_{\tau \in \mathcal{T}_{s}}\esssup_{\sigma \in \mathcal{T}_{s}}\mathcal{J}^{i,j}_{s}(\sigma,\tau) = \esssup_{\sigma \in \mathcal{T}_{s}}\essinf_{\tau \in \mathcal{T}_{s}}\mathcal{J}^{i,j}_{s}(\sigma,\tau),
	\end{equation}
	where,
	\begin{equation}
	\begin{split}
	\mathcal{J}^{i,j}_{s}(\sigma,\tau) \coloneqq {} & \mathbb{E}\biggl[\int_{s}^{\sigma \wedge \tau}f^{i,j}_{t}{d}t + \mathbf{1}_{\{\tau < \sigma\}}{U^{i,j}_{\tau}}(\boldsymbol{Y}) + \mathbf{1}_{\{\sigma \le \tau,\; \sigma < T\}}{L^{i,j}_{\sigma}}(\boldsymbol{Y}) \Bigm \vert \mathcal{F}_{s} \biggr] \\
	& + \mathbb{E}\bigl[h^{i,j}\mathbf{1}_{\{\sigma = \tau = T\}} \bigm \vert \mathcal{F}_{s} \bigr],
	\end{split}
	\end{equation}
	and $\boldsymbol{h}$, $\boldsymbol{f}$, $\boldsymbol{L}$ and $\boldsymbol{U}$ are the data for \eqref{eq:DRBSDE-Bilteral-Interconnected-Obstacles} (see Definition \ref{Definition:DRBSDE}).
	\vskip1em
	$(b)$ we have $Y^{i,j}_{s} = \mathcal{J}^{i,j}_{s}(\sigma^{i,j}_{s},\tau^{i,j}_{s})$ where $\sigma^{i,j}_{s} \in \mathcal{T}_{s}$ and $\tau^{i,j}_{s} \in \mathcal{T}_{s}$ are stopping times defined by,
	\begin{equation}\label{eq:Implicit-Obstacle-Time-Selector}
	\begin{cases}
	\sigma^{i,j}_{s}  = \inf\{s \le t \le T \colon Y^{i,j}_{t} = L^{i,j}_{t}(\boldsymbol{Y})\} \wedge T, \\
	\tau^{i,j}_{s} = \inf\{s \le t \le T \colon Y^{i,j}_{t} = U^{i,j}_{t}(\boldsymbol{Y})\} \wedge T,
	\end{cases}
	\end{equation}
	and we use the convention that $\inf\emptyset = +\infty$. Moreover, $\bigl(\sigma^{i,j}_{s},\tau^{i,j}_{s}\bigr)$ is a Nash equilibrium for the Dynkin game,
	\begin{equation}
	\mathcal{J}^{i,j}_{s}(\sigma,\tau^{i,j}_{s}) \le \mathcal{J}^{i,j}_{s}(\sigma^{i,j}_{s},\tau^{i,j}_{s}) \le \mathcal{J}^{i,j}_{s}(\sigma^{i,j}_{s},\tau) \quad \forall \sigma \in \mathcal{T}_{s} \text{ and } \tau \in \mathcal{T}_{s}.
	\end{equation}
\end{proposition}
\begin{proof}
	Recalling the ordering \eqref{eq:DRBSDE-Bilteral-Interconnected-Obstacles}-$(ii)$, the result follows from Proposition 2.2.1 of \cite{Hamadene2004}, for example.
\end{proof}

We will use Proposition \ref{Problem:Implicit-Dynkin-Game-Upper-Value} and a dynamic programming argument to first establish claim {\it (i)} of Theorem~\ref{Theorem:Verification}, then obtain {\it (ii)} and {\it (iii)} as corollaries. Since \eqref{eq:Verification-Ansatz} trivially holds when $s = T$, let $s \in [0,T)$ and $(i,j) \in \Gamma$ be arbitrary. Define a sequence $(\rho_{n},\gamma_{n})_{n \ge 0}$ as follows,
\begin{gather}\label{eq:Joint-Optimal-Switching-Control}
\rho_{0} = s, \quad \gamma_{0} = (i,j) \enskip \text{and for} \enskip n \ge 1,\\
\rho_{n} = \sigma^{\gamma_{n-1}}_{\rho_{n-1}} \wedge \tau^{\gamma_{n-1}}_{\rho_{n-1}}, \quad \gamma_{n} = \begin{cases}
\bigl(\mathcal{L}_{\rho_{n}}^{\gamma_{n-1}}(\boldsymbol{Y}),\gamma_{n-1}^{(2)}\bigr),& \text{on} \quad \mathcal{M}^{+}_{n} \\
\bigl(\gamma_{n-1}^{(1)}, \mathcal{U}^{\gamma_{n-1}}_{\rho_{n}}(\boldsymbol{Y})\bigr),& \text{on} \quad \mathcal{M}^{-}_{n} \\
\gamma_{n-1},& \text{otherwise}
\end{cases}
\end{gather}
where $\sigma^{\gamma_{n-1}}_{\rho_{n-1}}$ and $\tau^{\gamma_{n-1}}_{\rho_{n-1}}$ are defined using \eqref{eq:Implicit-Obstacle-Time-Selector} above, $\mathcal{L}^{\gamma_{n-1}}_{\rho_{n}}$ and $\mathcal{U}^{\gamma_{n-1}}_{\rho_{n}}$ are obtained from the switching selectors,
\begin{equation}\label{eq:Implicit-Obstacle-Mode-Selector}
\begin{cases}
\mathcal{L}^{i,j}_{t}(\boldsymbol{Y}) \in \argmax\limits_{i_{1} \neq i, i_{1} \in \Gamma^{1}}\{Y^{i_{1},j}_{t} - \hat{g}^{i,i_{1}}_{t}\}, \\
\mathcal{U}^{i,j}_{t}(\boldsymbol{Y}) \in \argmin\limits_{j_{1} \neq j, j_{1} \in \Gamma^{2}}\{Y^{i,j_{1}}_{t} + \check{g}^{j,j_{1}}_{t}\},
\end{cases}
\end{equation}
and for $n \ge 1$, $\mathcal{M}^{+}_{n}$ and $\mathcal{M}^{-}_{n}$ are the events,
\[
\begin{cases}
\mathcal{M}^{+}_{n} = \bigl\{\sigma^{\gamma_{n-1}}_{\rho_{n-1}} \le \tau^{\gamma_{n-1}}_{\rho_{n-1}},\;\sigma^{\gamma_{n-1}}_{\rho_{n-1}} < T\bigr\},\\
\mathcal{M}^{-}_{n} = \bigl\{\tau^{\gamma_{n-1}}_{\rho_{n-1}} < \sigma^{\gamma_{n-1}}_{\rho_{n-1}}\bigr\}.
\end{cases}
\]

\begin{lemma}\label{Lemma:Joint-Optimal-Control} Under the conditions of Theorem \ref{Theorem:Verification} we have $\gamma(\alpha^{*},\beta^{*}) \in \mathcal{G}^{i,j}_{s}$ and $Y^{i,j}_{s} = J^{i,j}_{s}(\gamma(\alpha^{*},\beta^{*}))$ a.s., where $\alpha^{*} = (\sigma^{*}_{n},\xi^{*}_{n})_{n \ge 0}$ and $\beta^{*} = (\tau^{*}_{n},\zeta^{*}_{n})_{n \ge 0}$ are sequences defined from $(\rho_{n},\gamma_{n})_{n \ge 0}$ as follows,
\begin{gather}\label{eq:Decomposition-Of-Joint-Control}
\sigma^{*}_{0} = \tau^{*}_{0} = s, \quad (\xi^{*}_{0},\zeta^{*}_{0}) = (i,j) \enskip \text{and for} \enskip n \ge 1,\\
\begin{cases}
\sigma^{*}_{n} = \inf\{t \ge  \sigma^{*}_{n-1} \colon u^{(1)}_{t} \neq \xi^{*}_{n-1} \} \wedge T, \quad \xi^{*}_{n} = u^{(1)}_{\sigma^{*}_{n}+}, \\
\tau^{*}_{n} = \inf\{t \ge  \tau^{*}_{n-1} \colon u^{(2)}_{t} \neq \zeta^{*}_{n-1} \} \wedge T, \quad \zeta^{*}_{n} = u^{(2)}_{\tau^{*}_{n}+},
\end{cases}
\end{gather}
where $u$ is defined using \eqref{eq:Optimal-Switching-Indicator-Continuous-Time}.
\end{lemma}
\begin{proof}
We begin by establishing that $\alpha^{*} \in \mathsf{A}^{i}_{s}$. The non-free loop property \eqref{eq:SwitchingCostNoArbitrage-2} prevents accumulation of the switching times $\rho^{*}_{n}$, in the sense that $\mathbb{P}(\{\rho^{*}_{n} < T\; \forall n \ge 0\}) = 0$ (see, for example, \cite[pp.~192--193]{Hamadene2012}). Since $\sigma^{*}_{n} \ge \rho_{n}$ for $n \ge 0$, it follows that $\mathbb{P}(\{\sigma^{*}_{n} < T\; \forall n \ge 0\}) = 0$. Also, the consistency property \eqref{eq:SwitchingCostNoArbitrage-1} ensures that it is not optimal for a single player to switch twice at the same instant, so we have $\sigma^{*}_{n} < \sigma^{*}_{n+1}$ on $\{\sigma^{*}_{n} < T\}$ for $n \ge 1$ (see \cite{Martyr2014b} or \cite{Hamadene2012}). By the construction of $\alpha^{*}$, noting that $u^{(1)}_{\sigma^{*}_{n}+}$ is $\mathcal{F}_{\sigma^{*}_{n}}$-measurable since $\mathbb{F}$ is right-continuous, the remaining parts of Definition \ref{Definition:Switching-Controls} are satisfied, and $\alpha^{*} \in \mathsf{A}^{i}_{s}$. Similarly $\beta^{*} \in \mathsf{B}^{j}_{s}$.

We now prove that $\gamma(\alpha^{*},\beta^{*}) \in \mathcal{G}^{i,j}_{s}$ by proceeding in a similar manner to \cite{Hamadene2012}. Using \eqref{eq:DRBSDE-Bilteral-Interconnected-Obstacles}-(i) and \eqref{eq:DRBSDE-Bilteral-Interconnected-Obstacles}-(iii) together with the construction of $\rho_1$ gives $\mathbb{P}$-a.s.,
\begin{align*}
Y^{i,j}_{s} & = \int_{s}^{\rho_{1}}f^{i,j}_{t}{d}t + h^{i,j}\mathbf{1}_{\{\rho_{1} = T\}} + Y^{i,j}_{\rho_{1}}\mathbf{1}_{\{\rho_{1} < T\}} + \int_{s}^{\rho_{1}}{d}K^{i,j,+}_{t} - \int_{s}^{\rho_{1}}{d}K^{i,j,-}_{t} - \int_{s}^{\rho_{1}}Z^{i,j}_{t}{d}B_{t}, \nonumber \\
& = \int_{s}^{\rho_{1}}f^{i,j}_{t}{d}t + h^{i,j}\mathbf{1}_{\{\rho_{1} = T\}} + Y^{i,j}_{\rho_{1}}\mathbf{1}_{\{\rho_{1} < T\}} - \int_{s}^{\rho_{1}}Z^{i,j}_{t}{d}B_{t}.
\end{align*}
By considering the first switch for either player we have
\begin{align*}
Y^{i,j}_{s} = {} & \int_{s}^{\rho_{1}}f^{i,j}_{t}{d}t + \Bigl(Y^{\gamma^{(1)}_{1},j}_{\sigma^{i,j}_{s}} - \hat{g}^{i,\gamma^{(1)}_{1}}_{\sigma^{i,j}_{s}}\Bigr)\mathbf{1}_{\{\sigma^{i,j}_{s} < T\}}\mathbf{1}_{\{\sigma^{i,j}_{s} \le \tau^{i,j}_{s}\}} + \Bigl(Y^{i,\gamma^{(2)}_{1}}_{\tau^{i,j}_{s}} - \check{g}^{j,\gamma^{(2)}_{1}}_{\tau^{i,j}_{s}}\Bigr)\mathbf{1}_{\{\tau^{i,j}_{s} <\, \sigma^{i,j}_{s}\}} \nonumber \\
& + h^{i,j}\mathbf{1}_{\{\rho_{1} = T\}} - \int_{s}^{\rho_{1}}Z^{i,j}_{t}{d}B_{t} \nonumber \\
= {} & \int_{s}^{\rho_{1}}f^{u_{t}}_{t}{d}t + Y^{\gamma_{1}}_{\rho_{1}}\mathbf{1}_{\{\rho_{1} < T\}} + h^{\gamma_{0}}\mathbf{1}_{\{\rho_{1} = T\}} - \Bigl[\hat{g}^{\gamma^{(1)}_{0},\gamma^{(1)}_{1}}_{\rho_{1}} - \check{g}^{\gamma^{(2)}_{0},\gamma^{(2)}_{1}}_{\rho_{1}}\Bigr] - \int_{s}^{\rho_{1}}Z^{u_{t}}_{t}{d}B_{t},
\label{eq:Verification-First-Step}
\end{align*}
(to account for the event $\{\rho_{1} = T\}$, recall that $\hat{g}^{i,i}_{t} = \check{g}^{j,j}_{t} = 0$).
Proceeding iteratively for $n = 1,\ldots,N$ we obtain by substitution
\begin{equation}\label{eq:Integrate-BSDE-Optimal-Control-1}
\begin{split}
Y^{i,j}_{s} = {} & \int_{s}^{\rho_{N}}f^{u_{t}}_{t}{d}t + \sum_{n = 1}^{N}h^{\gamma_{n-1}}\mathbf{1}_{\{\rho_{n} = T,\; \rho_{n-1} < T\}} - \sum_{n = 1}^{N}\Bigl[\hat{g}^{\gamma^{(1)}_{n-1},\gamma^{(1)}_{n}}_{\rho_{n}} - \check{g}^{\gamma^{(2)}_{n-1},\gamma^{(2)}_{n}}_{\rho_{n}} \Bigr] \\
& + Y^{\gamma_{N}}_{\rho_{N}}\mathbf{1}_{\{\rho_{N} < T\}} - \int_{s}^{\rho_{N}}Z^{u_{t}}_{t}{d}B_{t},
\end{split}
\end{equation}
from which we obtain	
\begin{align}\label{eq:Integrate-BSDE-Optimal-Control-2}
C^{\gamma(\alpha^{*},\beta^{*})}_{N} = {} & Y^{\gamma_{N}}_{\rho_{N}}\mathbf{1}_{\{\rho_{N} < T\}} - Y^{i,j}_{s} + \int_{s}^{\rho_{N}}f^{u_{t}}_{t}{d}t + \sum_{n = 1}^{N}h^{\gamma_{n-1}}\mathbf{1}_{\{\rho_{n} = T,\; \rho_{n-1} < T\}} \nonumber \\
& - \int_{s}^{\rho_{N}}Z^{u_{t}}_{t}{d}B_{t}.
\end{align}
Let $M^{u} = (M^{u}_{t})_{s \le t \le T}$ denote the stochastic integral $M^{u}_{t} = \int_{s}^{t}Z^{u_{r}}_{r}{d}B_{r}$, which is a well-defined square-integrable martingale on $[s,T]$ \cite{Chung2014}. Continuing from \eqref{eq:Integrate-BSDE-Optimal-Control-2} we have a.s.,
\begin{align}\label{eq:Integrate-BSDE-Optimal-Control-3}
\sup_{N \ge 1}\bigl|C^{\gamma(\alpha^{*},\beta^{*})}_{N}\bigr| \le {} & \int_{s}^{T}|f^{u_{t}}_{t}|{d}t + \max_{(i,j) \in \Gamma}|h^{i,j}| + |Y^{i,j}_{s}| + \max_{(i,j) \in \Gamma}\sup_{s \le t \le T}|Y^{i,j}_{s}| \nonumber \\
& + \sup_{s \le t \le T}|M^{u}_{t}|.
\end{align}
The right-hand side of \eqref{eq:Integrate-BSDE-Optimal-Control-3} is a square-integrable random variable, thereby proving $\gamma(\alpha^{*},\beta^{*}) \in \mathcal{G}^{i,j}_{s}$.

It is now straightforward to prove $Y^{i,j}_{s} = J^{i,j}_{s}(\gamma(\alpha^{*},\beta^{*}))$ a.s. by taking conditional expectations in \eqref{eq:Integrate-BSDE-Optimal-Control-1} then passing to the limit $N \to \infty$, which is justified since $\gamma(\alpha^{*},\beta^{*}) \in \mathcal{G}^{i,j}_{s}$,
\begin{align}\label{eq:Verification-Limiting-Step}
Y^{i,j}_{s} & = \mathbb{E}\left[\int_{s}^{T}f^{u_{t}}_{t}{d}t + h^{u_{T}} - \sum_{n = 1}^{\infty}\Bigl[\hat{g}^{\gamma^{(1)}_{n-1},\gamma^{(1)}_{n}}_{\rho_{n}} - \check{g}^{\gamma^{(2)}_{n-1},\gamma^{(2)}_{n}}_{\rho_{n}} \Bigr] \Bigm \vert \mathcal{F}_{s}\right] \nonumber \\
& = J^{i,j}_{s}(\gamma(\alpha^{*},\beta^{*})).
\end{align}
\end{proof}
For a given $\alpha = (\sigma_{n},\xi_{n})_{n \ge 0} \in \mathsf{A}_{s}^{i}$, let $\overline{\beta^{*}}(\alpha) = (\tau_{n},\zeta_{n})_{n \ge 0}$ be the control for player 2 defined similarly to \eqref{eq:Decomposition-Of-Joint-Control} with the sequence $(\rho_{n},\gamma_{n})_{n \ge 0}$ constructed by,
\begin{gather}\label{eq:One-Sided-Optimal-Switching-Control-Player-2}
\rho_{0} = s, \quad \gamma_{0} = (i,j) \enskip \text{and for} \enskip n \ge 1,\\
\rho_{n} = \sigma_{\check{r}_{n}} \wedge \check{\tau}_{n}, \quad \gamma_{n} = \begin{cases}
\bigl(\xi_{\check{r}_{n}},\gamma_{n-1}^{(2)}\bigr),& \text{on} \quad \check{\mathcal{M}}^{+}_{n} \\
\bigl(\gamma_{n-1}^{(1)}, \mathcal{U}^{\gamma_{n-1}}_{\rho_{n}}(\boldsymbol{Y})\bigr),& \text{on} \quad \check{\mathcal{M}}^{-}_{n} \\
\gamma_{n-1},& \text{otherwise}
\end{cases}
\end{gather}
where $\mathcal{U}^{\gamma_{n-1}}_{\rho_{n}}$ is obtained from \eqref{eq:Implicit-Obstacle-Mode-Selector}, $\check{\tau}_{n} \coloneqq \tau^{\gamma_{n-1}}_{\rho_{n-1}}$ for $n \ge 1$, $\{\check{r}_{n}\}_{n \ge 0}$ is defined iteratively by $\check{r}_{0} = 0$, $\check{r}_{1} = 1$ and for $n \ge 2$,
\[
\check{r}_{n} = \check{r}_{n-1} + \mathbf{1}_{\{\sigma_{\check{r}_{n-1}} \le\, \check{\tau}_{n - 1}\}},
\]
and for $n \ge 1$, $\check{\mathcal{M}}^{+}_{n}$ and $\check{\mathcal{M}}^{-}_{n}$ are the events,
\[
\begin{cases}
\check{\mathcal{M}}^{+}_{n} = \{\sigma_{\check{r}_{n}} \le \check{\tau}_{n},\;\sigma_{\check{r}_{n}} < T \},\\
\check{\mathcal{M}}^{-}_{n} = \{\check{\tau}_{n} < \sigma_{\check{r}_{n}}\}.
\end{cases}
\]
In an analogous manner using the lower switching selector $\boldsymbol{\mathcal{L}}(\boldsymbol{Y})$ in \eqref{eq:Implicit-Obstacle-Mode-Selector}, for each $\beta \in \mathsf{B}_{s}^{j}$ we define $\overline{\alpha^{*}}(\beta) \in \mathsf{A}_{s}^{j}$ for player 1. The following lemma points out key properties of $\overline{\alpha^{*}}$ and $\overline{\beta^{*}}$ utilised below to finish the proof of Theorem~\ref{Theorem:Verification}.

\begin{lemma}\label{Lemma:One-Sided-Optimal-Non-anticipative-Strategy-Player-2}
	\mbox{}
	\begin{enumerate}[(i)]
		\item We have $\overline{\alpha^{*}} \in \mathscr{A}^{i}_{s}$ and $\overline{\beta^{*}} \in \mathscr{B}^{j}_{s}$.
		\item We have
		\begin{equation}\label{eq:BSDE-Solution-Non-anticipative-Strategies}
		\esssup_{\alpha \in \mathcal{A}^{i}_{s}} J^{i,j}_{s}\bigl(\gamma(\alpha,\overline{\beta^{*}}(\alpha))\bigr) = Y^{i,j}_{s} = \essinf_{\beta \in \mathcal{B}^{j}_{s}} J^{i,j}_{s}\bigl(\gamma(\overline{\alpha^{*}}(\beta),\beta)\bigr).
		\end{equation}
	\end{enumerate}
\end{lemma}
\begin{proof}
	\mbox{}
	\vskip0.1em
	{\it Proof of (i)}:	We only show $\overline{\beta^{*}} \in \mathscr{B}^{j}_{s}$ since the proof that $\overline{\alpha^{*}} \in \mathscr{A}^{i}_{s}$ follows by similar arguments. Just as in the proof of Lemma \ref{Lemma:Joint-Optimal-Control}, the construction of $\overline{\beta^{*}}(\alpha)$ together with the no free-loop and consistency properties are sufficient to establish that $\overline{\beta^{*}}(\alpha) \in \mathsf{B}_{s}^{j}$ for each $\alpha \in \mathsf{A}_{s}^{i}$. Moreover, $\overline{\beta^{*}}$ satisfies the non-anticipative property in Definition~\ref{Definition:Switching-Strategies} by construction. Let $\alpha \in \mathcal{A}^{i}_{s}$ be given and let $\overline{\beta^{*}}(\alpha) = \beta = (\tau_{n},\zeta_{n})_{n \ge 0} \in \mathsf{B}^{j}_{s}$. To show that this control is square-integrable we will proceed as in the proof of Lemma \ref{Lemma:Joint-Optimal-Control}, to obtain that a.s.,
	\begin{align*}
	Y^{i,j}_{s} & = \int_{s}^{\rho_{1}}f^{i,j}_{t}{d}t + h^{i,j}\mathbf{1}_{\{\rho_{1} = T\}} + Y^{i,j}_{\rho_{1}}\mathbf{1}_{\{\rho_{1} < T\}} + \int_{s}^{\rho_{1}}{d}K^{i,j,+}_{t} - \int_{s}^{\rho_{1}}{d}K^{i,j,-}_{t} - \int_{s}^{\rho_{1}}Z^{i,j}_{t}{d}B_{t}, \nonumber \\
	& \ge \int_{s}^{\rho_{1}}f^{u_{t}}_{t}{d}t + h^{i,j}\mathbf{1}_{\{\rho_{1} = T\}} - \bigl[\hat{g}^{i,\gamma^{(1)}_{1}}_{\rho_{1}} - \check{g}^{j,\gamma^{(2)}_{1}}_{\rho_{1}} \bigr] + Y^{\gamma_{1}}_{\rho_{1}}\mathbf{1}_{\{\rho_{1} < T\}} - \int_{s}^{\rho_{1}}Z^{u_{t}}_{t}{d}B_{t},
	\end{align*}
where, in contrast to the proof of Lemma \ref{Lemma:Joint-Optimal-Control}, here $\alpha$ is arbitrary and so $\gamma^{(1)}_{1}$ is not necessarily optimal at time $\rho_{1}$. This means the inequality $L^{i,j}_{\rho_{1}}(\boldsymbol{Y}) \le Y^{i,j}_{\rho_{1}}$ must be enforced and the non-negative term $\int_{s}^{\rho_{1}}{d}K^{i,j,+}_{t}$ cannot be neglected. Proceeding iteratively for $n = 1,\ldots,N$ it follows that
\begin{align}\label{eq:Upper-Bound-Player-1-Suboptimal}
Y^{i,j}_{s} \ge {} & \int_{s}^{\rho_{N}}f^{u_{t}}_{t}{d}t + \sum_{n = 1}^{N}h^{\gamma_{n-1}}\mathbf{1}_{\{\rho_{n} = T,\; \rho_{n-1} < T\}} - \sum_{n = 1}^{N}\Bigl[\hat{g}^{\gamma^{(1)}_{n-1},\gamma^{(1)}_{n}}_{\rho_{n}} - \check{g}^{\gamma^{(2)}_{n-1},\gamma^{(2)}_{n}}_{\rho_{n}} \Bigr] \nonumber \\
& + Y^{\gamma_{N}}_{\rho_{N}}\mathbf{1}_{\{\rho_{N} < T\}} - \int_{s}^{\rho_{N}}Z^{u_{t}}_{t}{d}B_{t},
\end{align}
from which we obtain	
	\begin{align}\label{eq:Upper-Bound-Player-2-Response}
	\sum_{n = 1}^{N}\check{g}^{\gamma^{(2)}_{n-1},\gamma^{(2)}_{n}}_{\rho_{n}} \le {} & - \int_{s}^{\rho_{N}}f^{u_{t}}_{t}{d}t - \sum_{n = 1}^{N}h^{\gamma_{n-1}}\mathbf{1}_{\{\rho_{n} = T,\; \rho_{n-1} < T\}} + \sum_{n = 1}^{N}\hat{g}^{\gamma^{(1)}_{n-1},\gamma^{(1)}_{n}}_{\rho_{n}} \nonumber \\
	& + Y^{i,j}_{s} - Y^{\gamma_{N}}_{\rho_{N}}\mathbf{1}_{\{\rho_{N} < T\}} + \int_{s}^{\rho_{N}}Z^{u_{t}}_{t}{d}B_{t}.
	\end{align}
	Since $\mathbb{P}(\{\rho_{N} < T\; \forall N \ge 1\}) = 0$ the limits as $N \to \infty$ on both sides of \eqref{eq:Upper-Bound-Player-2-Response} are well defined. As the switching costs are non-negative we have
	\begin{equation}\label{eq:Bounds-Player-2-Response}
	0 \le \sum_{n \ge 1}\check{g}^{\zeta_{n-1},\zeta_{n}}_{\tau_{n}} \le - \int_{s}^{T}f^{u_{t}}_{t}{d}t - h^{u_{T}} + \sum_{n \ge 1}\hat{g}^{\xi_{n-1},\xi_{n}}_{\sigma_{n}} + Y^{i,j}_{s} + \int_{s}^{T}Z^{u_{t}}_{t}{d}B_{t}.
	\end{equation}
	Since $\alpha \in \mathcal{A}^{i}_{s}$, $Y^{i,j} \in \mathcal{S}^{2}_{c}$, $h^{i,j} \in L^{2}(\mathcal{F}_{T})$, and $f^{i,j}$, $Z^{i,j}$ belong to $\mathcal{H}^{2}$ for all $(i,j) \in \Gamma$, the random variable on the right-hand side of \eqref{eq:Bounds-Player-2-Response} belongs to $L^{2}$ and we conclude that the control $\beta$ is square-integrable.
	
	\vskip0.2em
	{\it Proof of (ii)}:
	We only show the first equality in \eqref{eq:BSDE-Solution-Non-anticipative-Strategies} as the second follows via similar arguments.
	We proceed by showing that for every $\alpha \in \mathcal{A}^{i}_{s}$ we have,
	\begin{equation}\label{eq:One-Sided-Verification-Inequality}
	Y^{i,j}_{s} \ge J^{i,j}_{s}\bigl(\gamma(\alpha,\overline{\beta^{*}}(\alpha))\bigr).
	\end{equation}
	Taking conditional expectations in \eqref{eq:Upper-Bound-Player-1-Suboptimal} above we get,
	\begin{align}\label{eq:One-Sided-Verification-N-th-Step}
	Y^{i,j}_{s} \ge {} & \mathbb{E}\left[\int_{s}^{\rho_{N}}f^{u_{t}}_{t}{d}t + \sum_{n = 1}^{N}h^{\gamma_{n-1}}\mathbf{1}_{\{\rho_{n} = T,\; \rho_{n-1} < T\}} - \sum_{n = 1}^{N}\Bigl[\hat{g}^{\gamma^{(1)}_{n-1},\gamma^{(1)}_{n}}_{\rho_{n}} - \check{g}^{\gamma^{(2)}_{n-1},\gamma^{(2)}_{n}}_{\rho_{n}} \Bigr] \Bigm \vert \mathcal{F}_{s}\right] \nonumber \\
	& + \mathbb{E}\left[Y^{\gamma_{N}}_{\rho_{N}}\mathbf{1}_{\{\rho_{N} < T\}} \vert \mathcal{F}_{s}\right].
	\end{align}
	Using {\it (i)} above we have $\gamma\bigl(\alpha,\overline{\beta^{*}}(\alpha)\bigr) \in \mathcal{G}^{i,j}_{s}$, so taking the limit $N \to \infty$ in \eqref{eq:One-Sided-Verification-N-th-Step} proves the inequality \eqref{eq:One-Sided-Verification-Inequality}.
	
	Next, for each integer $k \ge 0$ let $\alpha^{*}_{k}$ denote the truncation of the control $\alpha^{*}$ from Lemma~\ref{Lemma:Joint-Optimal-Control} to the first $k$ switches: $\alpha^{*}_{k} = \bigl(\sigma^{*}_{n},\xi^{*}_{n})_{0 \le n \le k}$ with $(T,\xi^{*}_{k})$ appended. Then $\alpha^{*}_{k} \in \mathcal{A}^{i}_{s}$ for each $k$ and $J^{i,j}_{s}\bigl(\gamma\bigl(\alpha^{*}_{k},\overline{\beta^{*}}(\alpha^{*}_{k})\bigr)\bigr) \to J^{i,j}_{s}\bigl(\gamma\bigl(\alpha^{*},\overline{\beta^{*}}(\alpha^{*})\bigr)\bigr)$ by the non-anticipative properties of $\overline{\beta^{*}}$ and as $\gamma\bigl(\alpha^{*},\overline{\beta^{*}}(\alpha^{*})\bigr) \in \mathcal{G}^{i,j}_{s}$. The claim
	\[
	\esssup_{\alpha \in \mathcal{A}^{i}_{s}} J^{i,j}_{s}\bigl(\gamma(\alpha,\overline{\beta^{*}}(\alpha))\bigr) = Y^{i,j}_{s},
	\]
	is then proved by passing to the limit $k \to \infty$ in,
	\[
	J^{i,j}_{s}\bigl(\gamma\bigl(\alpha^{*}_{k},\overline{\beta^{*}}(\alpha^{*}_{k})\bigr)\bigr) \le \esssup_{\alpha \in \mathcal{A}^{i}_{s}} J^{i,j}_{s}\bigl(\gamma(\alpha,\overline{\beta^{*}}(\alpha))\bigr) \le Y^{i,j}_{s},
	\]
	and using Lemma~\ref{Lemma:Joint-Optimal-Control}.
\end{proof}

\begin{proof}[Proof of Theorem~\ref{Theorem:Verification}]
	\mbox{}
\vskip0.1em
{\it Proof of (i) and (ii)}:
By construction we have $\alpha^{*} = \overline{\alpha^{*}}(\beta^{*})$ and $\beta^{*} = \overline{\beta^{*}}(\alpha^{*})$ so that, by
Lemma~\ref{Lemma:Joint-Optimal-Control},
\begin{equation}\label{eq:Equality-Decomposed-Control}
Y^{i,j}_{s} = J^{i,j}_{s}\bigl(\gamma(\alpha^{*},\beta^{*})\bigr) =  J^{i,j}_{s}\bigl(\gamma(\overline{\alpha^{*}}(\beta^{*}),\beta^{*})\bigr) = J^{i,j}_{s}\bigl(\gamma(\alpha^{*},\overline{\beta^{*}}(\alpha^{*}))\bigr),
\end{equation}
and by Lemma~\ref{Lemma:One-Sided-Optimal-Non-anticipative-Strategy-Player-2},
\[
\esssup_{\alpha \in \mathcal{A}^{i}_{s}} J^{i,j}_{s}\bigl(\gamma(\alpha,\overline{\beta^{*}}(\alpha))\bigr) = Y^{i,j}_{s} = \essinf_{\beta \in \mathcal{B}^{j}_{s}} J^{i,j}_{s}\bigl(\gamma(\overline{\alpha^{*}}(\beta),\beta)\bigr).
\]
Since $\overline{\beta^{*}}(\alpha) \in \mathcal{B}^{j}_{s}$ for every $\alpha \in \mathcal{A}^{i}_{s}$ and $\overline{\alpha^{*}}(\beta) \in \mathcal{A}^{i}_{s}$ for every $\beta \in \mathcal{B}^{j}_{s}$, almost surely we have,
\[
\hat{V}^{i,j}_{s} \coloneqq \essinf_{\beta \in \mathcal{B}^{j}_{s}}\esssup_{\alpha \in \mathcal{A}^{i}_{s}} J^{i,j}_{s}\bigl(\gamma(\alpha,\beta)\bigr) \le Y^{i,j}_{s} \le \esssup_{\alpha \in \mathcal{A}^{i}_{s}}\essinf_{\beta \in \mathcal{B}^{j}_{s}} J^{i,j}_{s}\bigl(\gamma(\alpha,\beta)\bigr) \eqqcolon \check{V}^{i,j}_{s},
\]
which completes the proof since $\hat{V}^{i,j}_{s} \ge \check{V}^{i,j}_{s}$ a.s.
\vskip0.2em
{\it Proof of (iii)}:
For all $\overline{\alpha} \in \mathscr{A}^{i}_{s}$ we have a.s.,
\begin{align*}
\essinf\limits_{\beta \in \mathcal{B}^{j}_{s}} J^{i,j}_{s}\bigl(\gamma(\overline{\alpha}(\beta),\beta)\bigr) & \le \essinf\limits_{\beta \in \mathcal{B}^{j}_{s}}\esssup_{\alpha \in \mathcal{A}^{i}_{s}}J^{i,j}_{s}\bigl(\gamma(\alpha,\beta)\bigr) \\
& = Y^{i,j}_{s} = \essinf\limits_{\beta \in \mathcal{B}^{j}_{s}} J^{i,j}_{s}\bigl(\gamma(\overline{\alpha^{*}}(\beta),\beta)\bigr),
\end{align*}
and the corresponding statement for $\overline{\beta^{*}}$ is proved analogously. Since $\overline{\alpha^{*}} \in \mathscr{A}^{i}_{s}$ and $\overline{\beta^{*}} \in \mathscr{B}^{j}_{s}$ the proof is complete.
\end{proof}

\begin{remark}
	In proving Theorem~\ref{Theorem:Verification} we established the following. For players 1 and 2 respectively there exist {\it non-anticipative strategies} $\overline{\alpha^{*}}$ and $\overline{\beta^{*}}$ as well as {\it controls} $\alpha^{*}$ and $\beta^{*}$ which satisfy the following,
	\begin{itemize}
		\item the controls $\alpha^{*}$, $\beta^{*}$ and non-anticipative strategies $\overline{\alpha^{*}}$, $\overline{\beta^{*}}$ are related by $\alpha^{*} = \overline{\alpha^{*}}(\beta^{*})$ and $\beta^{*} = \overline{\beta^{*}}(\alpha^{*})$;
		\item $\alpha^{*}$ and $\beta^{*}$ are jointly admissible;
		\item when player $2$ (the minimiser) uses the non-anticipative strategy $\overline{\beta^{*}}$, then the use of the control $\alpha^{*}$ by player 1 (the maximiser) gives the maximum possible value for the switching game over all  controls $\alpha$ such that $(\alpha,\overline{\beta^{*}}(\alpha))$ is jointly admissible, including all square-integrable controls $\alpha$;
		\item when player $1$ uses the non-anticipative strategy $\overline{\alpha^{*}}$, then the use of the control $\beta^{*}$ by player 2 gives the minimum possible value for the switching game over all controls $\beta$ such that $(\overline{\alpha^{*}}(\beta),\beta)$ is jointly admissible, including all square-integrable controls $\beta$;
		\item the strategies $\overline{\alpha^{*}}$ and $\overline{\beta^{*}}$ are best responses in the robust sense \cite{Krasovskii1988, Ball2002, Bayraktar2016robust}.
	\end{itemize}
\end{remark}

Let us emphasise that $\overline{\alpha^{*}}$ is not necessarily a best response strategy in the sense,
\[
J^{i,j}_{s}\bigl(\gamma(\overline{\alpha^{*}}(\beta),\beta)\bigr) = \esssup_{\alpha \in \mathcal{A}^{i}_{s}}J^{i,j}_{s}\bigl(\gamma(\alpha,\beta)\bigr) \quad \forall \beta \in \mathcal{B}^{j}_{s}, \\
\]
and correspondingly for $\overline{\beta^{*}}$. In the game with initial data $(s,i,j)$, for player 1 we can define a mapping $\overline{\alpha} \colon \mathcal{B}^{j}_{s} \to \mathcal{A}^{i}_{s}$ such that for each $\beta \in \mathcal{B}^{j}_{s}$ a.s.,
\[
J^{i,j}_{s}\bigl(\gamma(\overline{\alpha}(\beta),\beta)\bigr) \ge J^{i,j}_{s}\bigl(\gamma(\alpha,\beta)\bigr)\; \text{a.s.} \quad \forall \alpha \in \mathcal{A}^{i}_{s},
\]
but this mapping is generally not non-anticipative since its output $\overline{\alpha}(\beta)$ can depend on the entire trajectory corresponding to the input $\beta$. For example, define the following objective for player 1,
\begin{align}\label{eq:Player-1-Non-Anticipative-Objective}
\tilde{J}_{s}^{i}(\alpha;\beta) & = \mathbb{E}\left[\int_{s}^{T}\tilde{f}^{u^{(1)}_{t}}_{t}{d}t - \sum_{n=1}^{\infty}\hat{g}^{\xi_{n-1},\xi_{n}}_{\sigma_{n}} + \tilde{h}^{u^{(1)}_{T}}_{s,T} \Bigm \vert \mathcal{F}_{s}\right], \quad \alpha \in \mathcal{A}^{i}_{s}, \enskip \beta \in \mathcal{B}, \\
\tilde{V}_{s}^{i}(\beta) & = \esssup\limits_{\alpha \in \mathcal{A}_{s}^{i}}\tilde{J}_{s}^{i}(\alpha;\beta),\nonumber
\end{align}
where $u^{(k)}$, defined analogously to \eqref{eq:Optimal-Switching-Indicator-Continuous-Time}, indicates the current mode selected by player $k = 1,2$, and for $i \in \Gamma^{1}$ and $t \in [s,T]$, $\tilde{f}^{i}_{t} \coloneqq f^{i,u^{(2)}_{t}}_{t}$ and $\tilde{h}^{i}_{s,T} \coloneqq h^{i,u^{(2)}_{T}} + \sum_{n=1}^{\infty}\check{g}^{\zeta_{n-1},\zeta_{n}}_{\tau_{n}}\mathbf{1}_{\{\tau_{n} \ge s\}}$.
Using the results in \cite{Martyr2014b,Djehiche2017}, we can prove the existence of value processes $\bigl(\tilde{V}_{t}^{i}(\beta)\bigr)_{s \le t \le T}$, $i \in \Gamma^{1}$, and an optimal control in $\mathcal{A}^{i}_{s}$ for each $i \in \Gamma^{1}$. The non-anticipativity issue arises from the dependence of \eqref{eq:Player-1-Non-Anticipative-Objective} on the expected future rewards due to player 2's switching decisions.

\subsection{The diffusion framework}
Recall the process $X^{s,x}$ introduced in \eqref{sdex} where $(s,x)\in [0,T] \times \mathbb{R}^{k}$. Suppose that $b$ and $\sigma$ are deterministic continuous functions with values in $\mathbb{R}^{k}$ and $\mathbb{R}^{k \times d}$ respectively, Lipschitz with respect to $x$ uniformly in $t$. Consequently, the process $X^{s,x}$  exists and is unique (see \cite{karatzas-shreve}). Next assume that for any $(i,j)\in \Gamma$, $k\in \Gamma^1$, $\ell \in \Gamma^2$ and $t \in [s,T]$, 
$$
f^{i,j}_{t} = \bar f^{i,j}(t,X^{s,x}_t), h^{i,j}=\bar  h^{i,j}(X^{s,x}_T), \hat g^{i,k}_t= \hat {\bar  g}^{i,k}(t,X^{s,x}_t) \mbox{ and }\check g^{j,\ell}_t= \check {\bar g}^{j,\ell}(t,X^{s,x}_t) 
$$
where the functions $\bar f^{i,j}, \bar  h^{i,j}, \hat {\bar g}^{i,k} \mbox{ and }\check {\bar  g}^{j,\ell}$ are deterministic, continuous and of polynomial growth with respect to $x$. We then have:
\begin{theorem}[see \cite{Djehiche2017}] Assume that:

a) the functions $\bar  h^{i,j}$, $\hat {\bar  g}^{i,k}$ and $\check {\bar g}^{j,\ell}$, $(i,j)\in \Gamma$, $k\in \Gamma^1$, $\ell \in \Gamma^2$, verify the properties of positivity, consistency and non-free loop of Assumption 1.

b) The functions $\check {\bar g}^{j,\ell}$, $j,\ell\in \Gamma^2$ or $\hat {\bar g}^{i,k}$, $i,k\in \Gamma^1$ 
are ${\cal C}^{1,2}$ and their derivatives are of polynomial growth.
\medskip 

\noindent Then there exists a system of processes $(Y^{i,j},Z^{i,j},K^{i,j})_{(i,j)\in \Gamma}$ which satisfy \eqref{eq:DRBSDE-Bilteral-Interconnected-Obstacles} on $[s,T]$, and for any $(i,j)\in \Gamma$, $Y^{i,j}_t$, $t\in [s,T]$ verifies \eqref{eq:Verification-Ansatz} and \eqref{eq:Switching-Game-Value}. Moreover, there also exist deterministic continuous functions with polynomial growth $(v^{i,j}(s,x))_{(i,j)\in \Gamma}$ such that for any $(i,j)$ and $t\in [s,T]$, 
$$
Y^{i,j}_t=v^{i,j}(t,X^{s,x}_t)
$$ 
and $(v^{i,j}(s,x))_{(i,j)\in \Gamma}$ is the unique solution in viscosity sense of system \eqref{mainsyst-vi}.
\end{theorem}


\begin{thebibliography}{10}
	
	\bibitem{Robin1976}
	Robin M.
	\newblock {Some optimal control problems for queueing systems}.
	\newblock In: Wets RJB, editor. Stochastic Systems: Modeling, Identification
	and Optimization, II (Mathematical Programming Studies Vol. 6). vol.~6.
	Springer Berlin Heidelberg; 1976. p. 154--169.
	\newblock doi:10.1007/BFb0120749.
	
	\bibitem{Tang1993}
	Tang S, Yong J.
	\newblock {Finite horizon stochastic optimal switching and impulse controls
		with a viscosity solution approach}.
	\newblock Stochastics An International Journal of Probability and Stochastic
	Processes. 1993;45(3):145--176.
	\newblock doi:10.1080/17442509308833860.
	
	\bibitem{Lundstrom2013a}
	Lundstr{\"{o}}m NLP, Nystr{\"{o}}m K, Olofsson M.
	\newblock {Systems of variational inequalities in the context of optimal
		switching problems and operators of Kolmogorov type}.
	\newblock Annali di Matematica Pura ed Applicata. 2014;193(4):1213--1247.
	\newblock doi:10.1007/s10231-013-0325-y.
	
	\bibitem{Bismut1979}
	Bismut JM.
	\newblock {Contr{\^{o}}le de processus alternants et applications}.
	\newblock Zeitschrift f{\"{u}}r Wahrscheinlichkeitstheorie und Verwandte
	Gebiete. 1979;47(3):241--288.
	\newblock doi:10.1007/BF00535163.
	
	\bibitem{Morimoto1987}
	Morimoto H.
	\newblock {Optimal switching for alternating processes}.
	\newblock Applied Mathematics {\&} Optimization. 1987;16(1):1--17.
	\newblock doi:10.1007/BF01442182.
	
	\bibitem{Zabczyk1973}
	Zabczyk J.
	\newblock {Optimal control by means of switching}.
	\newblock Studia Mathematica. 1973;45:161--171.
	
	\bibitem{Zvonkin1971}
	Zvonkin AK.
	\newblock {On Sequentially Controlled Markov Processes}.
	\newblock Mathematics of the USSR-Sbornik. 1971;15(4):607--617.
	\newblock doi:10.1070/SM1971v015n04ABEH001565.
	
	\bibitem{Djehiche2009}
	Djehiche B, Hamad{\`{e}}ne S, Popier A.
	\newblock {A Finite Horizon Optimal Multiple Switching Problem}.
	\newblock SIAM Journal on Control and Optimization. 2009;48(4):2751--2770.
	\newblock doi:10.1137/070697641.
	
	\bibitem{Martyr2014b}
	Martyr R.
	\newblock {Finite-Horizon Optimal Multiple Switching with Signed Switching
		Costs}.
	\newblock Mathematics of Operations Research. 2016;41(4):1432--1447.
	\newblock doi:10.1287/moor.2016.0783.
	
	\bibitem{Chassagneux2011}
	Chassagneux JF, Elie R, Kharroubi I.
	\newblock {A note on existence and uniqueness for solutions of multidimensional
		reflected BSDES}.
	\newblock Electronic Communications in Probability. 2011;16:120--128.
	\newblock doi:10.1214/ECP.v16-1614.
	
	\bibitem{Elie2014}
	Elie R, Kharroubi I.
	\newblock {Adding constraints to BSDEs with jumps: an alternative to
		multidimensional reflections}.
	\newblock ESAIM: Probability and Statistics. 2014 jul;18:233--250.
	\newblock doi:10.1051/ps/2013036.
	
	\bibitem{hu-tang2}
	Hu Y, Tang S.
	\newblock {Multi-dimensional BSDE with oblique reflection and optimal
		switching}.
	\newblock Probability Theory and Related Fields. 2010 may;147(1-2):89--121.
	\newblock doi:10.1007/s00440-009-0202-1.
	
	\bibitem{Yong1990}
	Yong J.
	\newblock {Differential games with switching strategies}.
	\newblock Journal of Mathematical Analysis and Applications. 1990
	jan;145(2):455--469.
	\newblock doi:10.1016/0022-247X(90)90413-A.
	
	\bibitem{Yong1990a}
	Yong J.
	\newblock {A Zero-Sum Differential Game in a Finite Duration with Switching
		Strategies}.
	\newblock SIAM Journal on Control and Optimization. 1990 sep;28(5):1234--1250.
	\newblock doi:10.1137/0328066.
	
	\bibitem{Tang2007}
	Tang S, Hou Sh.
	\newblock {Switching Games of Stochastic Differential Systems}.
	\newblock SIAM Journal on Control and Optimization. 2007;46:900--929.
	\newblock doi:10.1137/050642204.
	
	\bibitem{hu-tang}
	Hu Y, Tang S.
	\newblock {Switching game of backward stochastic differential equations and
		associated system of obliquely reflected backward stochastic differential
		equations}.
	\newblock Discrete and Continuous Dynamical Systems. 2015
	may;35(11):5447--5465.
	\newblock doi:10.3934/dcds.2015.35.5447.
	
	\bibitem{Djehiche2017}
	Djehiche B, Hamad{\`{e}}ne S, Morlais MA, Zhao X.
	\newblock {On the equality of solutions of max–min and min–max systems of
		variational inequalities with interconnected bilateral obstacles}.
	\newblock Journal of Mathematical Analysis and Applications. 2017
	aug;452(1):148--175.
	\newblock doi:10.1016/j.jmaa.2017.02.025.
	
	\bibitem{Stettner1982b}
	Stettner L.
	\newblock {Zero-sum Markov games with stopping and impulsive strategies}.
	\newblock Applied Mathematics {\&} Optimization. 1982;9(1):1--24.
	\newblock doi:10.1007/BF01460115.
	
	\bibitem{Cosso2013}
	Cosso A.
	\newblock {Stochastic Differential Games Involving Impulse Controls and
		Double-Obstacle Quasi-variational Inequalities}.
	\newblock SIAM Journal on Control and Optimization. 2013;51(3):2102--2131.
	\newblock doi:10.1137/120880094.
	
	\bibitem{Cohen2015}
	Cohen SN, Elliott RJ.
	\newblock {Stochastic Calculus and Applications}.
	\newblock Probability and Its Applications. New York, NY: Springer New York;
	2015.
	\newblock doi:10.1007/978-1-4939-2867-5.
	
	\bibitem{Chung2014}
	Chung KL, Williams RJ.
	\newblock Introduction to stochastic integration.
	\newblock 2nd ed. Modern Birkh\"auser Classics. New York, NY:
	Birkh\"auser/Springer; 2014.
	\newblock doi:10.1007/978-1-4614-9587-1.
	
	\bibitem{Krasovskii1988}
	Krasovski\u{\i} NN, Subbotin AI, Subbotin AI.
	\newblock {Game-Theoretical Control Problems}.
	\newblock Springer Series in Soviet Mathematics. New York, NY: Springer New
	York; 1988.
	\newblock doi:10.1007/978-1-4612-3716-7.
	
	\bibitem{Ball2002}
	Ball JA, Chudoung J, Day MV.
	\newblock {Robust Optimal Switching Control for Nonlinear Systems}.
	\newblock SIAM Journal on Control and Optimization. 2002;41(3):900--931.
	\newblock doi:10.1137/S0363012900372611.
	
	\bibitem{Bayraktar2016robust}
	Bayraktar E, Cosso A, Pham H.
	\newblock {Robust Feedback Switching Control: Dynamic Programming and Viscosity
		Solutions}.
	\newblock SIAM Journal on Control and Optimization. 2016;54(5):2594--2628.
	\newblock doi:10.1137/15M1046903.
	
	\bibitem{Bismut1981}
	Bismut JM.
	\newblock {Convex inequalities in stochastic control}.
	\newblock Journal of Functional Analysis. 1981;42(2):226--270.
	\newblock doi:10.1016/0022-1236(81)90043-4.
	
	\bibitem{Bouchard2009}
	Bouchard B.
	\newblock {A stochastic target formulation for optimal switching problems in
		finite horizon}.
	\newblock Stochastics An International Journal of Probability and Stochastic
	Processes. 2009;81(2):171--197.
	\newblock doi:10.1080/17442500802327360.
	
	\bibitem{Bismut1977}
	Bismut JM.
	\newblock {Sur un probl{\`{e}}me de dynkin}.
	\newblock Zeitschrift f{\"{u}}r Wahrscheinlichkeitstheorie und Verwandte
	Gebiete. 1977;39(1):31--53.
	\newblock doi:10.1007/BF01844871.
	
	\bibitem{Hamadene2006}
	Hamad{\`{e}}ne S, Hassani M.
	\newblock {BSDEs with two reflecting barriers driven by a Brownian motion and
		Poisson noise and related Dynkin game}.
	\newblock Electronic Journal of Probability. 2006;11:121--145.
	\newblock doi:10.1214/EJP.v11-303.
	
	\bibitem{Martyr2014c}
	Martyr R.
	\newblock {Solving finite time horizon Dynkin games by optimal switching}.
	\newblock Journal of Applied Probability. 2016;53(04):957--973.
	\newblock doi:10.1017/jpr.2016.57.
	
	\bibitem{Pham2013}
	Pham T, Zhang J.
	\newblock {Some norm estimates for semimartingales}.
	\newblock Electronic Journal of Probability. 2013;18:1--26.
	\newblock doi:10.1214/EJP.v18-2406.
	
	\bibitem{Dumitrescu2014}
	Dumitrescu R, Quenez Mc, Sulem A.
	\newblock {Generalized Dynkin games and doubly reflected BSDEs with jumps}.
	\newblock Electronic Journal of Probability. 2016;21:1--32.
	\newblock doi:10.1214/16-EJP4568.
	
	\bibitem{Hamadene2004}
	Hamad{\`{e}}ne S, Hassani M.
	\newblock {BSDEs with two reflecting barriers : the general result}.
	\newblock Probability Theory and Related Fields. 2004;132(2):237--264.
	\newblock doi:10.1007/s00440-004-0395-2.
	
	\bibitem{Hamadene2012}
	Hamad{\`{e}}ne S, Morlais MA.
	\newblock {Viscosity Solutions of Systems of PDEs with Interconnected Obstacles
		and Switching Problem}.
	\newblock Applied Mathematics {\&} Optimization. 2013;67(2):163--196.
	\newblock doi:10.1007/s00245-012-9184-y.
	
	\bibitem{karatzas-shreve}
	Karatzas I, Shreve SE.
	\newblock {Brownian Motion and Stochastic Calculus}. vol. 113 of Graduate Texts
	in Mathematics.
	\newblock 2nd ed. New York, NY, USA: Springer New York; 1998.
	\newblock doi:10.1007/978-1-4612-0949-2.
	
\end{thebibliography}

\end{document}